\documentclass[11pt]{amsart}
\usepackage{amssymb,amsmath,amsthm,latexsym}
\usepackage[all]{xy}
\usepackage{amssymb,graphicx,color}
\usepackage{marvosym}
\usepackage{amsthm}
\usepackage{amscd}
\usepackage{xspace}

\newtheorem{theorem}{Theorem}

\newtheorem{corollary}[theorem]{Corollary}

\newtheorem{definition}[theorem]{Definition}

\newtheorem{lemma}[theorem]{Lemma}

\newtheorem{proposition}[theorem]{Proposition}

\newtheorem*{mnthm}{Main Theorem}

\newcommand\h{{\mathfrak h}}

\renewcommand\k{{\mathfrak k}}
\newcommand\m{{\mathfrak m}}
\newcommand\n{{\mathfrak n}}

\newcommand\hh{{\hat{\h}}}
\newcommand\hk{{\hat{\k}}}

\newcommand\hm{{\hat{\m}}}

\newcommand\cH{{\mathcal H}}

\newcommand\Aut{\operatorname{Aut}}
\newcommand\Stab{\operatorname{Stab}}
\newcommand\Fix{\operatorname{Fix}}

\renewcommand\paragraph[1]{{\bigskip\noindent{\bf #1}.}}

\newcommand{\cat}{{\upshape CAT(0)}\xspace}

\title{Ping Pong on \cat Cube Complexes}
\begin{document}
\author{Aditi Kar and Michah Sageev}
\begin{abstract} Let $G$ be a group acting properly and essentially on an irreducible, non-Euclidean finite dimensional CAT(0) cube complex $X$ without a global fixed point at infinity. We show that for any finite collection of simultaneously inessential subgroups $\{H_1, \ldots, H_k\}$  in $G$, there exists an element $g$ of infinite order such that $\forall i$, $\langle H_i, g\rangle \cong H_i * \langle g\rangle$. We apply this to show that any group, acting faithfully and geometrically on a non-Euclidean possibly reducible CAT(0) cube complex, has property $P_{naive}$ i.e. given any finite list $\{g_1, \ldots, g_k\}$  of elements from $G$, there exists $g$ of infinite order such that $\forall i$, $\langle g_i, g\rangle \cong \langle g_i \rangle *\langle g\rangle$. This applies in particular to the Burger-Mozes simple groups that arise as lattices in products of trees. The arguments utilize the action of the group on the boundary of \emph{strongly separated ultrafilters} and moreover, allow us to summarize equivalent conditions for the reduced $C^*$-algebra of the group to be simple. 
\end{abstract}

\maketitle

\parindent=0pt
\parskip = 12pt

\section{Introduction}

Felix Klein's Ping Pong Lemma is a widely used criterion for determining if a collection of group elements generate a non-abelian free subgroup and more generally, for constructing subgroups which are non-trivial free products. 
In this paper, we  employ the ping pong lemma in the setting of groups acting on  \cat cube complexes to construct subgroups which split as non-trivial free products, as described below in the Main Theorem. 

An action of a group $G$ on a \cat cube complex $X$ is said to be \emph{essential} if for any given orbit of $G$, there are orbit points which are arbitrarily deep inside any half space of $X$. A collection of groups $G_1, \ldots, G_k$ acting on $X$ are said to be \emph{simultaneously inessential} if there is a half space $\h$ and a vertex $v \in X$ such that $\cup_i G_i(v) \subset \h$. 
A large class of examples of simultaneously inessential subgroups arise when $G$ is Gromov hyperbolic and acts properly cocompactly on $X$ and  the $G_i$'s are a finite collection of  infinite index quasi-convex subgroups of $G$ (see Proposition \ref{QuasiconvexInessential}).  Our goal is the following. 
 
\begin{mnthm} Let $X$ be a finite dimensional, irreducible, non-Euclidean \cat cube complex and let $G$ be a group acting essentially and properly on $X$, without a global fixed point at infinity. Assume further that $G$ has no finite normal subgroup. Let $A_1,\ldots, A_n$ be a collection of simultaneously inessential subgroups of $G$. Then there exists $g\in G$ of infinite order, such that for each $i$, $ \langle g, A_i\rangle \cong \langle g\rangle * A_i.$
\label{main}
\end{mnthm}

If $H$ is a quasi-convex subgroup in a non-elementary hyperbolic group $G$, then Theorem \ref{main} holds and there exists $g \in G$ such that the subgroup generated by $g$ and $H$ is the free product $\langle g\rangle*H$; this was proved by Arzhantseva in \cite{GA}.


The key step in the proof of the Main Theorem that allows us to play ping pong is Proposition \ref{GoodHyperplane} which says that for any collection $A_1,\ldots, A_n$ of simultaneously inessential subgroups, one can find a half space $\h$ in $X$ such that $a\h$ is contained in the complement of $\h$, for all  nontrivial $a \in \cup_i A_i$. 

In the process of proving the Proposition, we construct a new ultrafilter boundary $S(X)$ built out of \emph{strongly separated} ultrafilters of $X$. The strongly separated ultrafilters have nice properties. For example, the median of three strongly separated ultrafilters is a vertex of $X$. We use this to show that the fixed set of every non-trivial element of the group has empty interior on the boundary. We summarize this into the proposition below, which is proved at the end of Section \ref{IrreducibleCase}.

\begin{proposition} \label{bound}
Let $X$ be a non-Euclidean irreducible CAT(0) cube complex $X$. Suppose a group $G$ is acting essentially on $X$ without a global fixed point at infinity. Then, the compact $G$-space $S(X)$ is minimal and strongly proximal and hence, a $G$-boundary. Moreover, if the action of $G$ on $X$ is proper and $G$ has no non-trivial finite normal subgroups, then the action of $G$ on $S(X)$ is topologically free. 
\end{proposition}

When $X$ splits into irreducible direct factors $X_1\times \ldots \times X_n$ and each factor $X_i$ is non-Euclidean then $S(X)$ decomposes as a direct product of the $S(X_i)$ and Proposition \ref{bound} naturally extends to the reducible case. A similar ultrafilter boundary was studied by Fernos in \cite{fernos}.

\subsection*{An Application of the Main Theorem : property $P_{naive}$} We use the Main Theorem to study property $P_{naive}$ for groups acting on \cat cube complexes. 

Property $P_{naive}$ was introduced by Bekka, de la Harpe and Cowling \cite{BCH} to study the ideal structure of group $C^*$-algebras. We give a brief introduction to $P_{naive}$ in section \ref{naive}. 

\begin{definition} A group $G$ has property $P_{naive}$ if for every finite subset $F \subset G$ there exists an element $y \in G$ of infinite order such that given $g \in F$, the subgroup $\langle g, y \rangle$ is canonically isomorphic to the free product $\langle g\rangle * \langle y\rangle$.  
\end{definition}

\begin{corollary}
Suppose a group $G$ is acting properly and cocompactly on a finite dimensional non-Euclidean CAT(0) cube complex. If $G$ has no non-trivial finite normal subgroups then $G$ has property $P_{naive}$.  
\end{corollary}

In the irreducible case, property $P_{naive}$ is a direct consequence of the Main Theorem, as given by Corollary \ref{naivecor}. When the underlying \cat cube complex is reducible, we prove property $P_{naive}$ for lattices in Aut$(X)$, where $X$ is locally finite, co-compact and has no Euclidean factor (see Theorem \ref{Products}). Examples of groups satisfying the hypotheses of Theorem \ref{Products}, which were not known up to now to satisfy $P_{naive}$, are the Burger-Mozes simple groups \cite{BurgerMozes2000}, which arise as lattices in products of trees.

The study of property $P_{naive}$ was initiated by Bekka, Cowling and de la Harpe as a means to establish $C^*$-simplicity of group $C^*$-algebras. Here, we use property $P_{naive}$ from the above Corollary and several previously known results to provide necessary and sufficient conditions for the reduced $C^*$-algebra of a \cat cube complex group to be simple. This last property is commonly referred to as $C^*$-simplicity. 

\begin{corollary}\label{equiv}
The following are equivalent for a group $G$ acting properly and co-compactly on a finite dimensional \cat cube complex $X$.
\begin{enumerate} 
\item $G$ has property $P_{naive}$.
\item $G$ is $C^*$-simple. 
\item Every non-trivial conjugacy class of $G$ is infinite.
\item The amenable radical of $G$ is trivial.
\item The $G$-action is faithful and $X$ is non-Euclidean.
\end{enumerate} 
\end{corollary} 

Recent research has yielded more sophisticated techniques for establishing $C^*$-simplicity. Kalantar and Kennedy \cite{KK}  have brought in dynamical techniques showing that a group $G$ is $C^*$-simple if and only if there exists a $G$-boundary on which the $G$-action is topologically free. Using Proposition \ref{bound}, we get an application of their Theorem to groups acting properly (not necessarily, co-compactly) on \cat cube complexes (refer to Proposition \ref{bound}) without a global fixed point at infinity.

Kalantar and Kennedy's methods were developed further by Breuillard, Kalantar, Kennedy and Ozawa \cite{BKKO}. Recall \cite[Theorem 3.1]{BKKO} which says that if a discrete group $G$ has countably many amenable subgroups, then $G$ is $C^*$-simple if and only if the amenable radical is trivial. In \cite{SageevWise}, Sageev and Wise showed that groups acting on finite dimensional \cat cube complexes satisfy the Tits Alternative so long as one knows the action is proper and there is a bound on the size of the finite subgroups. Their proof works equally well if the existence of a bound on the size of finite subgroups is replaced by the weaker condition that, every locally finite subgroup is finite. Therefore, if $G$ is acting properly on a finite dimensional CAT(0) cube complex and every locally finite subgroup of $G$ is finite, then the Tits Alternative for $G$ implies that every amenable subgroup is finitely generated virtually abelian. Consequently, if $G$ is countable, then $G$ can have only countably many amenable subgroups. We get the following interesting application of \cite[Theorem 3.1]{BKKO}. 

\begin{proposition}
Let $G$ be a countable discrete group such that every locally finite subgroup is finite. Suppose $G$ acts properly on a finite dimensional CAT(0) cube complex. Then, $G$ is $C^*$-simple if and only if its amenable radical is trivial. 
\end{proposition}

This generalizes Le Boudec's Proposition 3.2 from \cite{Adrien}, which deals with the case when $X$ is a product of trees. When the locally finite subgroups are not necessarily finite, groups acting properly on finite dimensional \cat cube complexes can have uncountably many amenable subgroups. For instance, one can make a direct sum of infinitely many copies of a finite cyclic group act properly on a tree. 

\subsection*{Acknowledgements} 
We would like to thank Moose, Luna and Shurjo, without whom this paper would have been possible. We would like to thank Emmanuel Breuillard, Pierre de la Harpe and the anonymous referee for their comments and suggestions for improving the paper. 
\section{Preliminaries}

In this section, we collect some relevant notions and results on \cat cube complexes, as well as introducing a few new notions. We refer the reader to \cite{CapraceSageev2011}, \cite{NevoSageev2013} and  \cite{Sageev2014} for details on the relevant background material.  In particular, we will assume familiarity with hyperplanes and halfspaces. We will always assume that $X$ is a finite dimensional \cat cube complex.  
We will use  $\h$ (and other gothic letters) to refer to a halfspace, $\h^*$ to refer to the complementary halfspace and $\hh$ to refer to a hyperplane.

\subsection{Essentiality}

A \cat cube complex is called \emph{essential} if every halfspace $\h$ contains arbitrarily large metric balls. This is the same as saying that every halfspace contains arbitrarily deep points: points arbitrarily far away from its bounding hyperplane. 

If $\Aut(X)$ acts on $X$ without a global fixed point either in $X$ or at infinity (the visual boundary), then $X$ contains an $\Aut(X)$ invariant essential core. Thus, it is reasonable to discuss only essential \cat cube complexes, and we shall assume this from now on. 

An action of a group $G$ on $X$ is said to be an \emph{essential action} if for any given orbit, there are orbit points arbitrarily deep inside every halfspace.  When $X$ is essential and the action is inessential there exists a halfspace $\h$ and a vertex $v$ such that $G(v)\subset \h$. A collection of subgroups  $G_1,...,G_n < \Aut(X)$ are said to be \emph{simultaneously inessential} if there exists halfspace $\h$ and a vertex $v$ in $X$ such that $\cup_i G_i(v)\subset \h$.

A large class of examples of simultaneously inessential subgroups arises in the context of hyperbolic groups. 

\begin{proposition} 
Let $G$ be a hyperbolic group which acts properly, cocompactly and essentially on a \cat cube complex $X$. Let $G_1,\ldots, G_n$ be a finite collection of infinite index quasiconvex subgroups. Then $G_1,\ldots, G_n$ are simultaneously inessential. 
\label{QuasiconvexInessential}
\end{proposition}

We delay the proof of Proposition \ref{QuasiconvexInessential} until Section \ref{IrreducibleCase}.

%

\subsection{Products}

We say that $X$ is reducible if it admits a decomposition as a product of two non-trivial \cat cube complexes. A finite dimensional \cat cube complex always admits a canonical decomposition as a product of irreducible complexes. 

If $X$ is essential then each irreducible factor of $X$ is also essential. Those irreducible factors that are not quasi-isometric to a real line are called \emph{non-Euclidean} factors. More explicitly, an irreducible, essential \cat cube complex is called non-Euclidean if it is not quasi-isometric to a real line. A (possibly reducible) essential \cat cube complex is called non-Euclidean if all of its factors are non-Euclidean. Essential,  irreducible, non-Euclidean complexes will be the subject of Section \ref{IrreducibleCase}. 

\subsection{Facing triples and strongly separated hyperplanes} The notion of a non-Euclidean \cat cube complex can be characterized in terms of facing triples of hyperplanes.  By a facing triple of hyperplanes we mean a pairwise disjoint triple of hyperplanes that bound halfspaces which are also pairwise disjoint. Equivalently, no hyperplane of the triple separates the other two from one another. We then have the following lemma.

\begin{lemma}[Facing Triples]
 Let $X$ be an essential, non-Euclidean \cat cube complex such that $\Aut(X)$ acts with no global fixed point at infinity. Then for every halfspace $\h$, there exists a facing triple $\hh, \hk,\hm$ with $\hk, \hm\subset \h$.
 \label{FacingTriples}
\end{lemma}

An important lemma for us regarding irreducible cube complexes involves strongly separated pairs. A pair of disjoint  hyperplanes $\hh$ and $\hk$ are called \emph{strongly separated} if there are no hyperplanes that intersect both $\hh$ and $\hk$. We will also refer to the corresponding nested pair of halfspaces $\h\subset\k$ as being strongly separated. We then have the following lemma.

\begin{lemma}[Strongly Separated Pairs]
 Let $X$ be an essential non-Euclidean \cat cube complex such that $\Aut(X)$ acts without a global fixed point at infinity. Then for every halfspace $\h$ there exists a halfspace $\k\subset\h$ such that $\hh$ and $\hk$ are strongly separated. 
 \label{StronglySeparated}
\end{lemma}

\subsection{Skewering} A halfspace $\h$ is said to be \emph{skewered} by an automorphism $g\in\Aut(X)$ if $g\h\subset\h$. We say that $g$ skewers the hyperplane $\hh$ if $g$ skewers $\h$ or $\h^*$. The  relevant lemma for us regarding skewering is the following. 

\begin{lemma}[Double Skewering]
Let $X$ be essential and $G$ act on $X$ either cocompactly or without a global fixed point at infinity. 
Then for every pair of halfspaces $\h\subset\k$, there exists $g\in G$ such that $g\k\subset\h$. 
\label{DoubleSkewering}
\end{lemma}

As a corollary of the Double Skewering Lemma, we have that every halfspace is skewered by some element. For given a halfspace $\h$, there exists some $\h\subset\k$ and then the element ensured by the Double Skewering Lemma skewers $\h$.

In fact, a generalization of this for products can be established. More precisely (Theorem C of \cite{CapraceSageev2011}), one can show the following.

\begin{theorem}
Let $X=X_1\times\ldots\times X_n$ be a product of infinite, locally compact \cat cube complexes such that $\Aut(X_i)$ acts cocompactly on $X_i$ for each $i$. Suppose that $G$ is a lattice in $\Aut(X)$. Suppose that $\h_i\subset\k_i$ are nested halfspaces in each factor $X_i$. Then there exists $g\in G$ which simultaneously double skewers these hyperplanes. That is to say, for each $i$, $g\k_i\subset\h_i$. 
\label{SimultaneousDoubleSkewering}
\end{theorem}

\subsection{The Roller Boundary}
As before, let $X$ be essential. We will consider here a certain part of the Roller boundary which will be useful to us (see \cite{Sageev2014} for basics on ultrafilters and the Roller boundary). Let $\cH$ denote the collection of halfspaces of $X$. Recall that an ultrafilter on $\cH$ is a subset $\alpha\subset \cH$ satisfying 
\begin{enumerate}
 \item (Choice) For each pair ${\h,\h^*}$, exactly one of $\h$ or $\h^*$ is in $\alpha$.
 \item (Consistency) If $\h\subset\k$ and $\h\in\alpha$ then $\k\in\alpha$.
\end{enumerate}

The collection of all ultrafilters ${\mathcal U}(X)$ has a natural topology induced by the Tychonoff topology on $2^\cH$. This has as a basis the collection of \emph{halfspace neighborhoods}, where a halfspace neighborhood is a subset of ${\mathcal U}(X)$ of the form 
\[U_\h\equiv\{\alpha\in{\mathcal U}(X)\vert \h\in\alpha\}\]

One can show that the collection of ultrafilters is then closed in $2^\cH$. 
The vertices of $X$ correspond to those ultrafilters satisfying the descending chain condition (DCC). 
The Roller Boundary is defined to be the complement in ${\mathcal U}(X)$ of the DCC ultrafilters. It is closed in ${\mathcal U}(X)$ as well and is therefore compact.

On the opposite side of the spectrum for ultrafilters,  we have what we call \emph{strongly separated} ultrafilters. 

\begin{definition}
An ultrafilter $\alpha$ is \emph{strongly separated} if there exists an infinite  nested sequence of halfspaces $\h_1\supset\h_2\ldots\in\alpha$ such that $\h_i$ and $\h_{i+1}$ are strongly separated. We call such a sequence of halfspaces a \emph{strongly separated sequence} of halfspaces.
\end{definition}

It is easy to see that there are strongly separated sequences of halfspaces, since by Lemma \ref{StronglySeparated},  any halfspace $\h$ contains a halfspace strongly separated from it. In fact, by employing the Facing Triple Lemma, there exist uncountably many strongly separated sequences. A key observation is that a strongly separated sequence uniquely determines an ultrafilter.

\begin{lemma}
For every strongly separated sequence of halfspaces $\h_1\supset\h_2\ldots$,  there exists a unique ultrafilter $\alpha$ such that
$\h_i\in\alpha$.
\label{filter}
\end{lemma}

\begin{proof}
We define an ultrafilter as follows.

\[\alpha=\{\h\vert \h_i\subset\h \text{ for infinitely many } i \}\]

By definition $\h_i\in\alpha$ for each $i$. We are left to check that $\alpha$ satisfies the two conditions necessary for an ultrafilter  (choice and consistency) and then that it is unique. Any given hyperplane $\hh$ may intersect at most one of the $\hh_i$'s. It follows that exactly one of the halfspaces $\h$, $\h^*$ contains infinitely many $\h_i$'s, thus precisely one of $\h,\h^*$ is in $\alpha$. The consistency condition is immediate since if infinitely many $\h_i$ satisfy $\h_i\subset\h$ and $\h\subset\k$ then $\h_i\subset\k$ for infinitely many $i$. 

To see uniqueness, let $\beta$ be an ultrafilter such that $\h_i\in\beta$ for all $i$. Then for any $\h\in\beta$, observe that $\hh$ may intersect at most one $\hh_i$. Consequently, either $\h$ contains infinitely many $\h_i$'s or $\h^*$ contains infinitely many $\h_i$'s. Choose one such $\h_i$. Since $\h,\h_i\in\beta$, by the consistency condition we have that $\h_i\subset\h$ (and not $\h_i\subset\h^*$). This means that $\h\in\alpha$. So $\alpha$ and $\beta$ make the same choices  for each pair $\h,\h^*$ and hence $\alpha=\beta$.
\end{proof}

We define $S(X)$ to be the closure in ${\mathcal U}$  of the collection of strongly separated ultrafilters. It is a compact subspace of the Roller Boundary. 

Next we see that strongly separated utrafilters behave nicely with respect to medians. Recall that given three ultrafilters $\alpha, \beta, \gamma$, the \emph{median} of $\alpha, \beta$ and $\gamma$ is defined as \[med(\alpha,\beta,\gamma)\equiv(\alpha\cap\beta)\cup(\beta\cap\gamma)\cup(\gamma\cap\alpha).\]

\begin{lemma}
 Let $\alpha,\beta,\gamma$ be distinct strongly separated ultrafilters. Then the $med(\alpha,\beta,\gamma)$ satisfies DCC and hence is a vertex of $X$. 
 \label{MedianInSpace}
\end{lemma}

\begin{proof}
We need to show that $\mu=med(\alpha,\beta,\gamma)$ satisfies the descending chain condition (see Figure \ref{Fig:MedianInSpace}). Suppose that $\h_1\supset\h_2\ldots$ is an infinite sequence of \nobreak{halfspaces} such that $\h_i\in\mu$. Then after passing to a subsequence, we may assume that $\h_i\in\alpha\cap\beta$ for all $i$. Since $\alpha$ and $\beta$ are distinct  strongly separated ultrafilters, there exist $\h\in\alpha$ and $\k\in\beta$ such that $\h\cap\k=\emptyset$ and $\hh$ and $\hk$ are strongly separated. Since $\h_i\in\alpha$, we have that $\h_i\cap\h\not=\emptyset$ and
\begin{figure}[h]
\includegraphics[width=.70\textwidth]{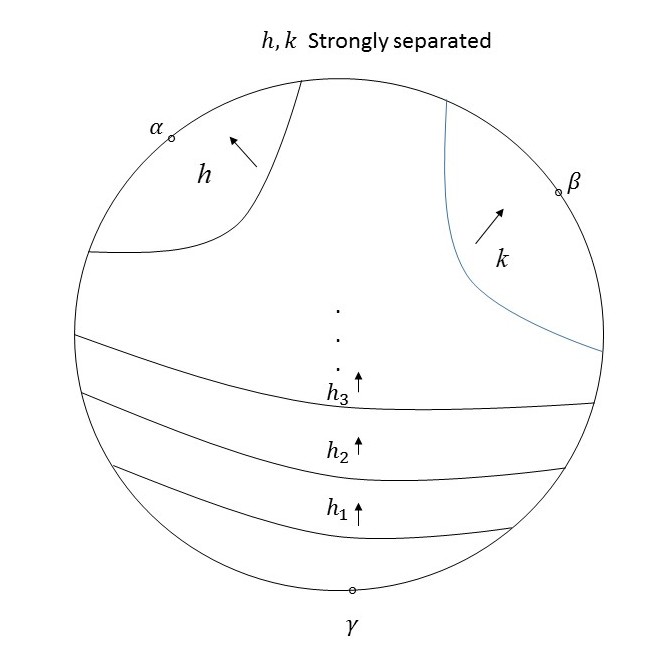}
 \caption{The median of strongly separated ultrafilters satisfies DCC.}
 \label{Fig:MedianInSpace}
\end{figure}
 $\h_i\cap\k\not=\emptyset$. But if $\{\h_i\}$ is an infinite descending sequence of hyperplanes, we must have that for $i$ sufficiently large,   $\h_i\subset\h$ or $\hh_i\cap\hh\not=\emptyset$.
 Similarly, for $i$ sufficiently large, we must have $\h_i\subset\k$ or $\hh_i\cap\hk\not=\emptyset$.
 But this contradicts the fact that $\hh$ and $\hk$ are strongly separated. 
  \end{proof}

We will also need the following lemma telling us that halfspace neighborhoods form a basic collection of open neighborhoods for the strongly separated ultrafilters. 

\begin{lemma}
Let $U\subset S(X)$ be an open neighborhood of $\alpha\in S(X)$, where $\alpha$ is a strongly separated ultrafilter. Then there exists a halfspace $\h$ such that $\alpha\in (U_\h\cap S(X))\subset U$.
\label{HalfSpaceNbhds}
\end{lemma}

\begin{proof}
Since the halfspace neighborhoods $U_\h$ serve as a collection of sub-basic open sets for the topology on ${\mathcal U}(X)$ and hence of $S(X)$, it suffices to prove this when  $U$ is a finite intersection of halfspace neighborhoods of $\alpha$. That is, we assume that there exist halfspaces $\h_1,\ldots,\h_n$ such that $U=\cap U_{\h_i} \cap S(X)$. Since $\alpha$ is a strongly separated ultrafilter, there exists a strongly separated sequence $\k_1\supset\k_2\ldots$ with $\k_i\in\alpha$. For each $\h_i$, we then know that there exists a tail of the strongly separated sequence contained in $\h_i$. Consequently, there exists a single $\k_j$ such that $\k_j\subset\h_i$ for all $i$. We then have that $\alpha\in U_{\k_j}\cap S(X)\subset U$ as required. 
\end{proof}

\subsection{Ping Pong}

We will use the following version of the Ping-Pong Lemma.

\begin{lemma}[Ping-Pong Lemma]
Let $S$ be a set and let $G$ be a group acting on $S$. Let $H,K<G$ be subgroups of $G$. Suppose that there exist two disjoint subsets $U,V\subset S$ such that for all for all $1\not=h\in H$, we have 
$hU\subset V$ and for all $1\not=k\in K$, $kV\subset U$. Then $<H,K>\cong H*K$.
\label{PingPong}
\end{lemma}

\section{Irreducible complexes}
\label{IrreducibleCase}

In all that follows, we will assume that $X$ is a  finite dimensional, irreducible, essential, non-Euclidean \cat cube complex, and that $G$ is a group acting on $X$  essentially, properly, and without global a fixed point at infinity. We also assume that $G$ has no finite normal subgroup. 
 
\begin{theorem}[Main Theorem]
 Let $A_1,\ldots, A_n$ be a collection of simultaneously inessential subgroups of $G$. Then there exists $g\in G$ of infinite order, such that for each $i$, 
$$ \langle g, A_i\rangle \cong \langle g\rangle * A_i $$
\label{MainTheorem}
\end{theorem}

\begin{corollary}
\label{naivecor}
Suppose that a group $G$ is acting on a finite-dimensional irreducible non-Euclidean \cat cube complex $X$. If the action of $G$ on $X$ is essential, proper and has no global fixed point at infinity, and, $G$ has no non-trivial finite normal subgroups then $G$ has property $P_{naive}$. 
\end{corollary}

First of all, we will need the following lemma. 

\begin{lemma}
Suppose that $a\in G$ is nontrivial. Then $\Fix(a)\subset S(X)$ has empty interior. 
\label{EmptyInterior}
\end{lemma}

\begin{proof}
 Suppose that $a$ is non-trivial and fixes an open subset $U\subset S(X)$. By Lemma \ref{HalfSpaceNbhds},
 there exists a half space $\h$ such that the halfspace neighborhood $U_\h\subset S(X)$.  
 Consider three strongly separated ultrafilters in $U_\h$ and let $v$ denote their median. By Lemma \ref{MedianInSpace}, the ultrafilter $v$ is a vertex in $X$.
 Since the action is essential, there exists $g\in G$ such that $g$ skewers $\h$,  so that $g\h\subset \h$. 
By the Lemmas  \ref{StronglySeparated} and  \ref{DoubleSkewering}, we may further assume that $g\hh$ and $\hh$ are strongly separated. 

 We now consider the elements $a_n \equiv g^{-n} a g^n$. Let $\h_n = g^{-n}\h$. Note that by our choice of $g$ above, the sequence $\{\h_n^*\}$ is a strongly separated sequence of halfspaces. 

 Note that $a_n$ fixes $U_n\equiv~U_{\h_n}$.  Since $v\in\h\subset g^{-n}\h$ it follows that $v$ is the median of three points contained in $U_n$, and therefore $a_n v= v$. By the properness of the action, there are only finitely many possibilities for $a_n$, so that we may pass to a subsequence of $\{a_n\}$ such that $a_n=b$ for all $n$. We then have $\bigcup_n U_n\subset \Fix(b)=\{y\in S(X) \vert by=y\}$. Because the action is proper, the kernel of the action on $S(X)$ is a finite normal subgroup and because  $G$ has no finite normal subgroup, we  have $Fix(b)\not= S(X)$. But now $\Fix(b)$ is closed. So there exists a halfspace $\k$ such that $U_\k\subset S(X)-\Fix(b)$. Consequently, we  have that $\hk\cap \h_n^*$, for all $n$. But this is a contradiction, since
 $\{\h_n^*\}$ is a strongly separated sequence of halfspaces. 
 \end{proof}

The key to proving the main theorem is the following proposition, which will allow us to play ping-pong. 

\begin{proposition}
 Let $A_1,\ldots A_n$ be a collection of simultaneously inessential subgroups of $G$. Then there exists a halfspace $\k$ in $X$, such that 
 $a\k\subset\k^*$ for all non-trivial $a\in \bigcup_i A_i$. 
 \label{GoodHyperplane}
\end{proposition}

\begin{proof}
To avoid writing indices, we will first give the proof for the case of a single subgroup $A$ and then later explain how this is done for finitely many subgroups. 

We will construct a combinatorial convex hull for $A(v)$, where $v$ is some vertex of $X$. For a halfspace $\h$, let $C(\h)$ denote the carrier of $\h$, namely the union of cubes that intersect $\h$ non-trivially. It is easy to see that that $C(\h)$ is a convex subcomplex of $X$ (see \cite{Haglund2008}). Now, given a halfspace $\h$ such that $A(v)\subset\h$, we define \[C_\h=\bigcap_{a\in A} C(a\h).\]

The inessentiality assumption tells us that there exists such an $\h$, and since $C_\h$ is the intersection of convex subcomplexes, it is convex. Also, $C_\h$ is invariant under $A$. 

\noindent{\bf Remark.} It is convex in both the usual CAT(0) sense but also in the $\ell_1$ sense: every combinatorial edge-geodesic between vertices in $C_\h$ remains in $C_\h$. 

Choose some halfspace $\k_1\subset\h^*$ such that $\hh$ and $\hk_1$ are strongly separated.

We observe that every hyperplane which intersects $C_\h$ does not intersect $\hk_1$, since $C_\h\subset\h$ and $\hh$ and $\hk_1$ are strongly separated.

\begin{figure}[h]
\includegraphics[width=.90\textwidth]{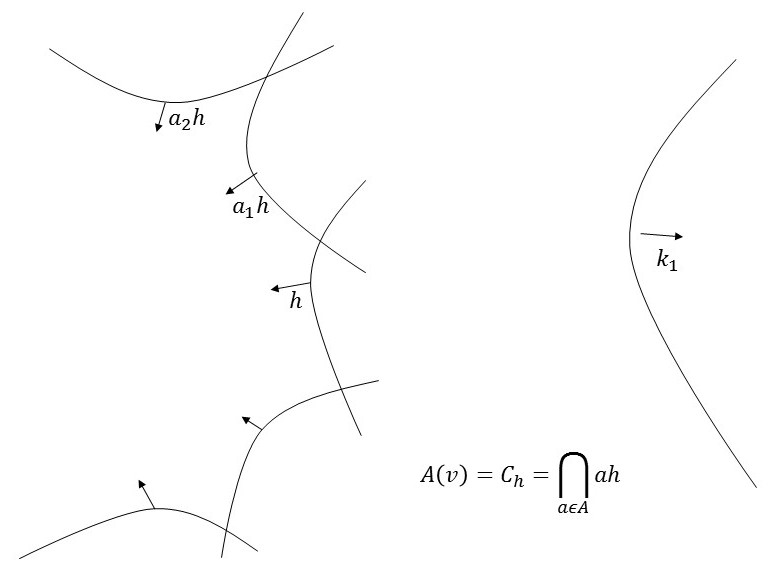}
 \caption{A convex hull for $A(v)$.}
 \label{Fig:Hull1}
\end{figure}

Now we consider the natural combinatorial projection of $\hk_1$ onto $C_h$. Namely, consider all the hyperplanes intersecting $C_\h$. As observed, for every such hyperplane $\hm$, we have $\hk_1\subset \m$ or $\hk_1\subset\m^*$. This thus defines an ultrafilter on the collection of hyperplanes meeting $C_\h$, since it is a choice of halfspaces which satisfies the standard consistency conditions necessary for an ultrafilter. It also satisfies the DCC condition (see, for example,  \cite{Sageev2014}). Thus, it determines a vertex $w$ in $C_\h$. This is the unique vertex of $C_\h$  that can be joined by a path to $\hk_1$ without crossing any hyperplane that meets $C_\h$.

Note that for any $a\in A$, $a\hk_1$ does not intersect any hyperplane that intersects $C(\h)$. This is because if it did, say $a\hk_1\cap\hm\not=\emptyset$, then by applying $a^{-1}$, we find that $\hk_1\cap a^{-1}(\hm)\not=\emptyset$. But by invariance of $C(\h)$ under $A$, we have that $a^{-1}(\hm)\cap C(\h)\not=\emptyset$, contradicting the strong separation of $\hk_1$ and $\hh$. 

Thus $a\hk_1$ projects to a vertex in $C(\h)$, just as $\hk_1$ does. 
Now by the naturality of this construction, we have that for each $a\in A$, the  translate $a\hk_1$ projects to $aw$. But if $a\hk_1\cap \hk_1\not=\emptyset$ it must project to $w$ as well. 

\begin{figure}[h]
\includegraphics[width=.90\textwidth]{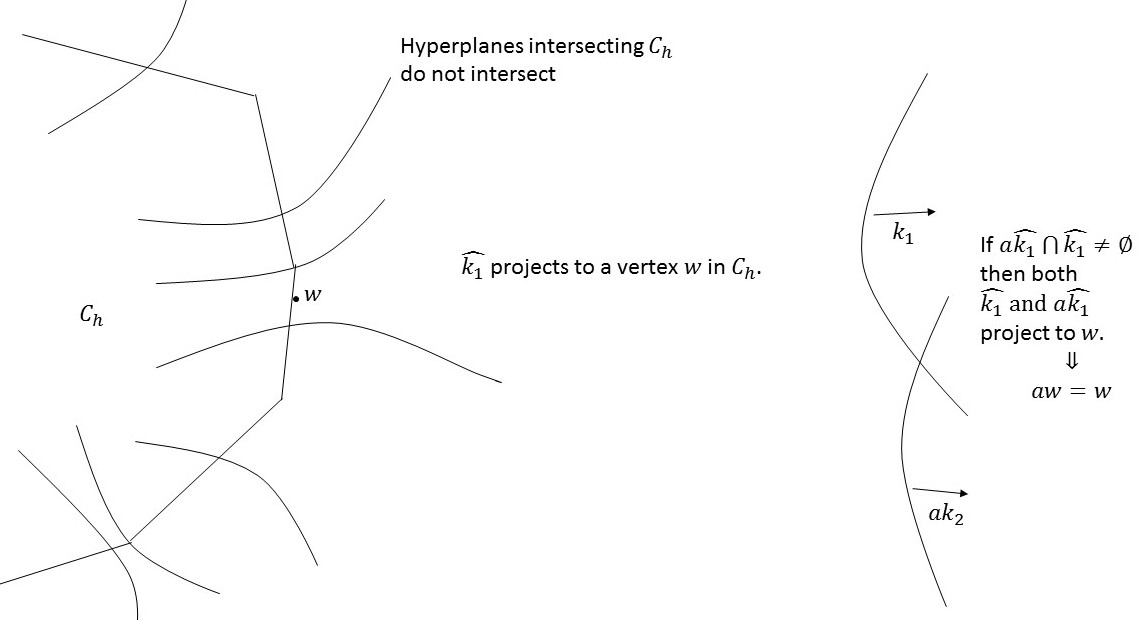}
 \caption{The projection of $\hk_1$ onto $C_\h$.}
 \label{Fig:Hull2}
\end{figure}

This tells us that 

$$S = \{a\in A\vert a\hk_1\cap\hk_1\not=\emptyset\}\subset \Stab(w)$$

By the properness of the action, we get that $S$ is finite. For all elements $a\not\in S$, we have $a\k\subset\k^*$, as required. We are thus left to prove the proposition for the elements of $S$. 

Let $U_{\k_1}$ denote the open subset of $S(X)$ determined by $\k_1$. By Lemma \ref{EmptyInterior} we can find a point $b\in U_{\hk_1}$ which is not fixed by any element of $S$. Since $S$ is finite, there exists a neighborhood $U\subset U_{\k_1}$ of $b$ such that $U\cap aU=\emptyset$ for any $a\in S$. Since every open neighborhood contains a halfspace neighborhood, we have a halfspace $\k_2\subset\k_1$ such that $aU_{\k_2}\cap U_{\k_2} = \emptyset$ for all $a\in S$. Thus, for any $a\in S$, we have that 
$a\k_2\subset \k_2^*$ for any $a\in S$. Since this is already true for $\k_1$ for all other elements of $A$, the hyperplane $\hk_2$ is the desired hyperplane.

To show the proposition for the case of finitely many subgroups $A_1,\ldots,A_n$ which are simultaneously inessential, we start with a hyperplane $\h$ such that $\bigcup _i A_i(v)\subset\h$. Taking $\hk_1$ as above
we see that the set of elements $S$  of $\bigcup_i A_i$ which carry $\hk_1$ to a hyperplane meeting $\hk_1$ is finite. 
We then construct, as in the previous paragraph a halfspace $\k_2\subset\k_1$ such that $a\k_2\subset\k_2^*$ for any element of $S$. This $\hk_2$ is the desired hyperplane. 
\end{proof}

\begin{proof}[Proof of Theorem \ref{MainTheorem}]
By Proposition \ref{GoodHyperplane}, there exists a halfspace $\k$ such that $a\k\subset\k^*$ for all $a\in \bigcup_i A_i$.
We need to find our $g\in G$ which plays ping-pong with every $A_i$. 

By the Facing Triples Lemma, there exists a pair of disjoint halfspaces $\m$ and $\n$ with $\m\cup \n\subset \k$.
By the Double Skewering Lemma, there exists $g\in G$ such that $g\m^*\subset\n$.

We now construct two disjoint subsets $U$ and $V$ of $X$, such that $aU\subset V$ for all non-trivial $a\in A_i$ and $g^n V\subset U$ for all $n\not=0$ and for all $i$.  This will then give the result by the Ping Pong Lemma. For each $i$, we define 

$$U=\bigcup_{a\not=1\in A_i} a\k \text{\ \ and\ \ } V = \m\cup g\m^*$$

\begin{figure}[h]
\includegraphics[width=.90\textwidth]{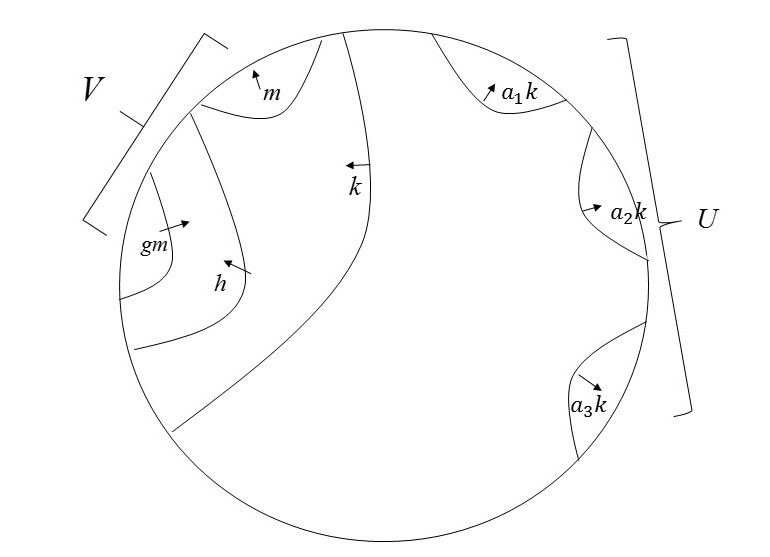}
 \caption{The construction of the ping pong pair.}
 \label{Fig:MainArgument}
\end{figure}

Note that by construction, for each $a\not=1\in A_i$, we have $a\k\subset U$, so we have $aV\subset U$. For each $n\not=0$, we obtain 
$g^n(g\m)\subset\m$ or $g^n\m^*\subset g\m^*$. Since $U\subset\m^*$ and $U\subset g\m$, we have that $g^n U\subset V$, as required. 
\end{proof}

We now prove Proposition \ref{QuasiconvexInessential}.

\begin{proof}
We prove the proposition by induction on $n$. Let $v$ be a vertex of $X$. In  \cite{SageevWise2015} (Proposition 3.3), it is shown that if $H$ is a quasiconvex subgroup of $G$, there exists a number $C>0$ and a universal number $D>0$ (depending only on the dimension of the complex), such that if $w$ is a vertex with $d(w, H(v))>C$, then there exists a hyperplane $\hh$ separating $w$ and $H(v)$ and $d(w,\hh)<D$. 

 Since points arbitrarily far away from $H(v)$ are guaranteed to exist when $H$ is of infinite index, this implies that the the proposition in the case $n=1$. 

We know assume that $G_1,\ldots,G_{n-1}$ are simultaneously inessential. Let $\h$ be a halfspace such that $G_i(v)\subset \h^*$ for $i=1,\ldots,n-1$. Note that since $v\subset \h^*$, the orbit $G_n(v)$ is not entirely contained in $\h$. We consider a halfspace $\k\subset\h$ such that $\hk$ is strongly separated from $\hh$. Since $G_n$ is of infinite index, there exists a vertex $w\in\k$ such that $d(w, G_n(v))>C$. We also choose $w$ such that $d(w,\hk)>D$. Now we apply the above again and conclude that there exists a hyperplane $\hm$ separating $w$ and $G_n(v)$ and such that $d(w,\hm)<D$. Let $\m$ be the halfspace associated to $\hm$ such that $w\in\m$ and $G_n(v)\subset\m^*$. Since $d(w,\hm)<D$, we have that $\hm\cap\k\not=\emptyset$. Since $\hk$ and $\hh$ are strongly separated, we thus have that $\hm\cap\hh=\emptyset$. Moreover, since $G_n(v)\cap\h^*\not=\emptyset$, we must have $\m\subset\h$ and not $\m^*\subset\h$. It follows
that $G_i(v)\subset\h$ for all $1=1,\ldots,n$, as required. 
\end{proof}

We complete the section with a proof of Proposition \ref{bound}, which says that $S(X)$ is a $G$-boundary, on which, if conditions are favourable, $G$ acts topologically freely. 

\begin{proof}[Proof of Proposition \ref{bound}]
Let $X$ be a non-Euclidean irreducible CAT(0) cube complex $X$. Suppose $G$ is acting essentially on $X$ without a global fixed point at infinity. Then, we claim that $S(X)$ is a $G$-boundary. We first show that the compact $G$-space $S(X)$ is minimal: given $\alpha \in S(X)$ and $U \subset S(X)$ open, there exists $g\in G$ such that $g\alpha \in U$. By Lemma \ref{HalfSpaceNbhds}, there exists some halfspace $\h$ such that $(U_\h\cap S(X))\subset U$. If $\h \in \alpha$ then we can take $g=1$. Suppose then $\h \notin \alpha$. By the Flipping Lemma \cite{CapraceSageev2011}, there exists $g \in G$ such that $g\h^* \subset \h$ and for this $g$, $\h \in g\alpha$. This implies $g\alpha \in (U_\h\cap S(X))\subset U$. 

We now show that the $S(X)$ is proximal: for any pair $\alpha, \beta \in S(X)$ of points, there exists a point $\gamma \in S(X)$ such that for every open neighbourhood $U$ of $\gamma$ there exists $g \in G$ such that $g\alpha, g\beta \in U$. Choose a strongly separated ultrafilter $\gamma$ which is distinct from both $\alpha$ and $\beta$. Let $U$ be any open set containing $\gamma$. Note that $S(X)$ is Hausdorff and so we can find an open set $V$ that contains $\gamma$ but does not contain $\alpha$ and $\beta$. The open set $U\cap V$ contains $\gamma$ and by Lemma \ref{HalfSpaceNbhds}, contains a half-space neighbourhood $U_\h$. Now, $\h^* \in \alpha, \beta $, so we use the Flipping Lemma to find $g \in G$ such that $g\h^* \subset \h$. Then, $\h \in g\alpha, g\beta$ and therefore $g\alpha, g\beta \in U$. 

To ensure that the proximal minimal $G$ space $S(X)$ is strongly proximal, we need to check that $S(X)$ has contractible neighbourhoods. Let $\alpha$ be a strongly separated ultrafilter and let $\h$ be a halfspace contained in $\alpha$. We claim that the open neighbourhood $V:=U_\h \cap S(X)$ of $\alpha$ is contractible i.e. there exists $\beta \in S(X)$ such that every open neighbourhood of $\beta$ contains a translate of $V$. Choose $\beta$ to be any strongly separated ultrafilter distinct from $\alpha$ and let $U$ be an open set containing $\beta$. As before, choose a halfspace $\k$ such that $\k \in \beta$, $\k \subset \h^*$ and $U_\k \cap S(X) \subset U$. Use the Flipping Lemma to choose $g\in G$ such that $g \k^* \subset \k$. Then, $gV \subset U$. 

This shows that $S(X)$ is a minimal and strongly proximal compact $G$-space. Lemma \ref{EmptyInterior} verifies that the action is topologically free whenever the action of $G$ on $X$ is proper and $G$ has no non-trivial finite normal subgroups. 
\end{proof}


\section{Property $P_{naive}$ and $C^*$-simplicity}\label{naive}
Recall that a group $G$ has property $P_{naive}$ if for every finite subset $F \subset G$ there exists an element $y \in G$ of infinite order such that given $g \in F$, the subgroup $\langle g, y \rangle$ is isomorphic to the free product $\langle g\rangle * \langle y\rangle$.  

The simplest example of a group possessing property  $P_{naive}$ is a non-abelian free group $F_n$. Property $P_{naive}$ was introduced by Bekka, Cowling and de la Harpe as part of their programme to study simplicity of group $C^*$-algebras \cite{BCH}. Non-elementary hyperbolic groups have property $P_{naive}$; this was proved for torsion-free groups by de la Harpe, and further generalized to relatively hyperbolic groups in \cite{ArzhantsevaMinasyan}. In \cite{BCH}, the authors established $P_{naive}$ for Zariski dense subgroups of connected simple Lie groups with $\mathbb{R}$-rank 1 and trivial center. More recently, property $P_{naive}$ was studied by Tal Poznansky in the context of linear groups: he proved that every Zariski-dense subgroup of a semisimple algebraic group (over any field), satisfies a weak version of property $P_{naive}$ \cite[Lemma 2.3]{Poznansky}. 

Here, we study conditions under which groups acting on \cat cube complexes have property $P_{naive}$. When the underlying complex is irreducible, property $P_{naive}$  follows from the Main Theorem and is recorded as Corollary \ref{naivecor} above. 

\subsection{Products}

In the case of products, we prove a result in a more restricted setting, namely that of lattices in $\Aut(X)$, where $X$ is a locally finite, cocompact cube complex. 

\begin{theorem}\label{Products}
Let $X$ be a locally finite, cocompact \cat cube complex with no Euclidean factors; let $G$ be a lattice in $\Aut(X)$ with no non-trivial finite normal subgroup. Then $G$ satisfies $P_{naive}$.
\end{theorem}

\begin{proof}
Let $X=\prod_k X_k$ be the decomposition of $X$ into irreducible factors. Let $g_1,\ldots, g_n$ denote a finite collection of elements of $G$. 

We first observe that for each $i$, the action of $<g_i>$ on each irreducible factor of $X$ is inessential. This is simply because $<g_i>$ is cyclic and each factor is non-euclidean. 

Secondly, we observe that since the action of $<g_i>$ is proper, there exists a factor of $X$ on which the action of $<g_i>$ is proper. Otherwise, for each factor $X_k$ there exists an integer $n_k$ such that $g^{n_k}$ fixes the ball of radius $R$ in $X_k$. Taking $N=\prod_k n_k$, we obtain an $N$ such that $<g_i^N>$ fixes the ball of radius $R$ in each $X_k$, which in the case that $g_i$ is infinite cyclic, would contradict the properness of the action of $<g_i>$ on $X$. 

For each factor $X_k$ for which $<g_i>$ acts properly on $X_k$,  Proposition \ref{GoodHyperplane} insures that there exists a halfspace $\h_k$, such that $a\h_k\subset \h_k^*$ for all 
$a\in <g_i>$. (If for some $k$, there are no such $g_i$'s, we choose $\h_k$ arbitrarily.) Following the proof of Theorem \ref{MainTheorem}, for each such $k$, we then choose halfspaces $\m_k$ and $\n_k$, so that $\m_k\cup \n_k\subset \h_k$. 

Now we apply Theorem \ref{SimultaneousDoubleSkewering} to conclude that there exists $g\in G$ such that $g\m_k^*\subset\n_k$ simultaneously for all $k$. The construction now of $U$ and $V$ for the application of the Ping Pong Lemma proceeds as in the proof of Theorem \ref{MainTheorem}.

More precisely, we need to show that $<g,g_i> \equiv <g> * <g_i>$. Given such an $i$, Let $X_k$ denote an irreducible component on which $<g_i>$ acts properly. Then set 

$$U=\bigcup_{a\not=1\in <g_i>} a\h_k \text{\ \ and\ \ } V = \m_k\cup\n_k$$

Then we obtain $aV\subset U$ for any $a\not=1$ and we have $g^nU\subset V$ for any $n\not=0$, as required. 
\end{proof}

\subsection{Infinite conjugacy classes}
Corollary \ref{naivecor} and Theorem \ref{Products} allow us to determine necessary and sufficient conditions for a \cat cube complex group to be $C^*$-simple. $C^*$-simple groups are often \textit{icc}: a group is icc if the conjugacy class of every non-identity element is infinite. We will first identify the collection of \cat cubical groups which are icc. 

\begin{proposition}\label{icc} 
If a group $G$ acts properly and co-compactly on a \cat cube complex then $G$ is icc if and only if no finite index subgroup of $G$ contains a non-trivial virtually abelian normal subgroup. 
\end{proposition}

\begin{proof}[Proof of Proposition] Suppose that $G$ is not icc. Let $H$ be the collection of all elements $g \in G$ such that the conjugacy class of $g$ is finite. It is easy to check that $H$ is a characteristic subgroup of $G$. Let $L$ be a subgroup generated by finitely many elements $x_1, \ldots, x_k$ of $H$. For each $i$, the centralizer of $x_i$ in $L$ is a subgroup of finite index in $L$. Consequently, the centre of $L$, which is the intersection of the centralizers of the $x_i$'s has finite index in $L$. This implies that each finitely generated subgroup of $H$ is virtually abelian. As every virtually abelian subgroup must stabilize a flat and the dimension of flats in $X$ is bounded, $H$ is forced to be virtually abelian. This shows, if $G$ has a non-trivial finite conjugacy class, then $G$ contains a non-trivial virtually abelian normal subgroup.  

Suppose now that a finite index subgroup $\Gamma$ of $G$ contains a virtually abelian normal subgroup $K$. If $K$ is finite and $g$ is a non-trivial element of $K$, then the conjugacy class $\{xgx^{-1}\ |\ x \in G\}$ of $g$ is contained in $\cup_{t \in G/\Gamma} tKt^{-1}$. Evidently, every conjugacy class of $K$ is finite and so, $G$ cannot be icc. If $K$ is infinite, then replace $K$ by a characteristic subgroup $K'$ which is free abelian of finite rank. The action of $\Gamma$ on $K$ by conjugation fixes $K'$ and so, $\Gamma$ normalizes $K'$. The homomorphism from $\Gamma$ to $Aut(K')\cong GL(n,\mathbb{Z})$ has finite image (in fact, it lies inside $O(n)\cap GL(n,\mathbb{Z}))$ and so, a finite index subgroup of $\Gamma$ (and hence, of $G$) that centralizes $K'$. Clearly, the conjugacy class of every element of $K'$ in $G$ is finite and $G$ cannot be icc.  
\end{proof}

The amenable radical of a group $G$, written $A_G$ is the largest amenable normal subgroup of $G$. As amenability is closed under extensions, the amenable radical exists and is easily shown to be a characteristic subgroup of $G$. Suppose a group $G$ has a finite index subgroup that contains a normal virtually abelian subgroup $K$. Then, passing to a normal finite index subgroup, we can assume that $G$ has a normal subgroup $H$ of finite index such that $A_H \neq 1$. As $A_H$ is characteristic in $H$, it is normal in $G$ and it follows, $A_G \neq 1$. Hence, the triviality of $A_G$ implies that $G$ has no finite index subgroups containing normal virtually abelian subgroups. 

In groups acting geometrically on \cat cube complexes the converse is true: the amenable radical is trivial if no finite index subgroup of $G$ has normal virtually abelian subgroups $\neq 1$. This is because, \cat cubical groups satisfy the Tits Alternative \cite[Main Theorem]{SageevWise} and the amenable radical is virtually abelian. Therefore $A_G \neq 1$ implies $G$ has a normal virtually abelian subgroup. To summarise, we have the following equivalence. 

\begin{lemma} \label{radical}
Suppose that a group $G$ is acting properly on a \cat cube complex and $G$ has a bound on the size of its finite subgroups. Then the amenable radical $A_G$ is trivial iff $G$ has no finite index subgroups with normal non-trivial virtually abelian subgroups. 
\end{lemma} 

The presence of virtually abelian subgroups inside finite index subgroups of $G$ is directly related to the existence of Euclidean factors in the Cartan decomposition of the underlying space. 

\begin{lemma} Suppose a group $G$ is acting geometrically and faithfully on a CAT(0) cube complex $X$. If $X$ has a Euclidean factor, then some finite index subgroup of $G$ contains a non-trivial virtually abelian normal subgroup. In particular, the amenable radical of $G$ is non-trivial. 
\label{euclidean}
\end{lemma} 

\begin{proof}
As $G$ is acting geometrically, $X$ is finite dimensional and moreover by passing to an essential core, we may assume that the $G$-action on $X$ is essential. Now, if $X$ is irreducible, then $X$ is Euclidean, meaning, it is quasi-isometric to the real line. In this case $G$ itself is virtually infinite cyclic. If $X$ is reducible, then it has a Cartan decomposition into irreducible factors. We have $X \cong X_P \times X_E$, where $X_E$ is the Euclidean part of $X$. Then, by Corollary 2.8 from \cite{NevoSageev2013}, there is a finite index subgroup $H$ of $G$ such that $H= H_E \times H_P$, where $H_E$ acts properly and co-compactly on $X_E$. This implies that $H_E$ is virtually abelian and so, a finite index subgroup contains a non-trivial virtually abelian normal subgroup. 
\end{proof}

\subsection*{$C^*$ simple groups. }
Let $G$ be a countable discrete group and let $\ell^2 G$ be the Hilbert space of square-summable functions on $G$. The group $G$ acts on $\ell^2G$ via its left regular representation as follows.  $$\lambda_g(f)(h)= f(g^{-1}h),\ \forall g,h\in G.$$ 
The map $g \mapsto \lambda_g$ gives an injection of $G$ into the space of bounded linear operators $\mathcal{B}(\ell^2 G)$. The closure of the linear span of image $\{\lambda_g\ :\ g \in G\}$ in the operator norm is called the reduced $C^*$-algebra of $G$ and written, $C^*_r(G)$. 

A countable group is said to be $C^*$-simple if $C^*_r(G)$ is a simple algebra, i.e. $C^*_r(G)$ has no non-trivial two-sided ideals. The reduced $C^*$-algebra carries information about the representation theory of the group. One can show that simplicity of the algebra $C^*_r(G)$ is equivalent to the following restriction on the representation theory of $G$: every unitary representation of $G$ which is weakly contained in the left regular representation of $G$ is actually equivalent to it. This means that a group which is both amenable and $C^*$-simple must be the trivial group. This statement in turn generalizes to the fact that a $C^*$-simple group cannot have non-trivial normal amenable subgroups.

Many geometric classes of groups have been shown to be $C^*$-simple. These include all free products (except the infinite dihedral group), non-soluble subgroups of $\mathrm{PSL}_2(\mathbb{R})$, torsion-free non-elementary hyperbolic groups and mapping class groups of surfaces. More  generally, acylindrically hyperbolic groups are $C^*$-simple \cite{DGO}.

A group acting geometrically on an irreducible \cat cube complex has enough rank one elements to make it acylindrically hyperbolic, using results from \cite{sisto}. So groups acting geometrically on irreducible \cat cubical groups are $C^*$-simple. However a group acting geometrically on a non-trivial product of irreducibles is not acylindrically hyperpbolic (for example, irreducible lattices in products of trees). Here, we apply our theorems on property $P_{naive}$ to show that even in this setting, a group $G$ acting properly and co-compactly on a \cat cube complex is $C^*$-simple. We summarize this as follows. 

\begin{theorem}[Corollary \ref{equiv}] Suppose that a group $G$ is acting properly and co-compactly on a \cat cube complex $X$. The following are equivalent. 
\begin{enumerate} 
\item $G$ is $C^*$ simple. 
\item $G$ is icc.
\item No finite index subgroup of $G$ has a non-trivial virtually abelian normal subgroup. 
\item the amenable radical of $G$ is trivial. 
\item The $G$-action is faithful and $X$ is non-Euclidean.
\item $G$ has property $P_{naive}$.
\end{enumerate} 
\end{theorem}

\proof The implications (1)$\Rightarrow$(2) and (6)$\Rightarrow$(1) are well-known, see \cite{BCH}. Proposition \ref{icc} establishes the equivalence of (2) and (3). Lemma \ref{radical} shows (3) and (4) are equivalent. That (4) implies (5) follows from Lemma \ref{euclidean}. 

The hypothesis that $G$ acts properly and co-compactly implies that $G$ is finitely presented and moreover, $X$ is finite-dimensional. The kernel of the action is finite whenever the action is proper and so if the amenable radical is trivial, the action is faithful. Now, to deduce (6) from (5), after passing to an essential core if needed, we apply Corollary \ref{naivecor} and Theorem \ref{Products}. 
\endproof

\small{A. Kar, University of Southampton, a.kar@soton.ac.uk \\ 
M. Sageev, Technion - Israel Institute of Technology, sageevmtx.technion.ac.il }

\end{document}